\documentclass[12pt,fleqn,leqno]{amsart}

\usepackage{amsfonts,amssymb}
\usepackage[bb=boondox]{mathalfa} 
\usepackage{enumitem}
\usepackage[svgnames]{xcolor}
\usepackage[colorlinks,linkcolor=blue,citecolor=Green]{hyperref}
\usepackage{titlesec}
\usepackage{tikz-cd}
\usetikzlibrary{decorations.pathmorphing}
\usetikzlibrary{snakes}

\titleformat{\section}[block]
 {\bfseries}
 {\thesection.}
 {\fontdimen2\font}
 {}
 
\setlist{noitemsep}
\setenumerate{labelindent=\parindent,label=\upshape{(\alph*)}}

\newtheorem{theorem}{Theorem}[section]

\newtheorem{proposition}[theorem]{Proposition}

\theoremstyle{definition}
\newtheorem{remark}[theorem]{Remark}
\newtheorem{example}[theorem]{Example}
\newtheorem{question}{Question}

\DeclareMathOperator{\N}{\mathbb{N}}
\DeclareMathOperator{\R}{\mathbb{R}}
\DeclareMathOperator{\uhr}{\upharpoonright} 

\renewcommand{\emptyset}{\varnothing}
\numberwithin{equation}{section}

\begin{document}

\author[V. Gutev]{Valentin Gutev}

\address{Department of Mathematics, Faculty of Science, University of
  Malta, Msida MSD 2080, Malta}

\email{\href{mailto:valentin.gutev@um.edu.mt}{valentin.gutev@um.edu.mt}}



\title[Simultaneous Extension of Uniformly Continuous
Functions]{Simultaneous Extension of Continuous and Uniformly
  Continuous Functions}

\begin{abstract}
  The first known continuous extension result was obtained by
  Lebes\-gue in 1907. In 1915, Tietze published his famous extension
  theorem generalising Lebesgue's result from the plane to general
  metric spaces. He constructed the extension by an explicit formula
  involving the distance function on the metric space. Thereafter,
  several authors contributed other explicit extension formulas. In
  the present paper, we show that all these extension constructions
  also preserve uniform continuity, which answers a question posed by
  St.\ Watson. In fact, such constructions are simultaneous for
  special bounded functions. Based on this, we also refine a result of
  Dugundji by constructing various continuous (nonlinear) extension
  operators which preserve uniform continuity as~well.
\end{abstract}

\date{\today}
\maketitle

\section{Introduction}

In his 1907 paper \cite{zbMATH02644079} on Dirichlet's problem,
Lebesgue showed that for a closed subset $A\subset \R^2$ and a
continuous function $\varphi:A\to \R$, there exists a continuous
function $f:\R^2\to \R$ with $f\uhr A=\varphi$.  Here, $f$ is commonly
called a \emph{continuous extension} of $\varphi$, and we also say
that $\varphi$ can be \emph{extended continuously}.\medskip

In 1915, Tietze \cite{Tietze1914a} generalised Lebesgue's result for
all metric spaces.

\begin{theorem}[Tietze, 1915]
  \label{theorem-special-fn-Tietze-v3:1}
  If $(X,d)$ is a metric space and $A\subset X$ is a closed set, then
  each bounded continuous function $\varphi:A\to \R$ can be extended
  to a continuous function $f:X\to \R$.
\end{theorem}

Tietze gave two proofs of Theorem
\ref{theorem-special-fn-Tietze-v3:1}, the second of which was based on
the following explicit construction of the extension. For a nonempty
closed set $A\subset X$ and a continuous function
$\varphi:A\to [1,2]$, he defined a continuous extension $f:X\to \R$ of
$\varphi$ by
\begin{equation}
  \label{eq:special-fn-Tietze:1}
  f(p)=
  \sup_{a\in A}\frac{\varphi(a)}{\left(1 +
      \left[d(a,p)\right]^2\right)^{\frac1{d(p,A)}}},\quad
  \text{whenever $p\in X\setminus A$.}
\end{equation}
Here, $d(p,A)=\inf\{d(p,a):a\in A\}$ is the \emph{distance} to the set
$A$, and it is assumed from the context that
$f\uhr A= \varphi$.\medskip

Nowadays, Theorem \ref{theorem-special-fn-Tietze-v3:1} is commonly
called \emph{Tietze's extension theorem}. The history of this theorem
is fascinating. In 1916, in his book \cite{poussin:1916}, de la
Vall\'{e}e Poussin gave a proof of this theorem for $X=\R^n$.  The
book of Carath\'{e}odory \cite{caratheodory:1918} contains another
proof of Tietze's extension theorem for Euclidean spaces, it was
credited to H. Bohr and works for general metric spaces, see Section
\ref{sec:bohrs-extens-constr}. In 1918, Brouwer \cite{MR1511921} also
gave an alternative proof of Tietze's extension theorem for Euclidean
spaces. In 1919, Hausdorff \cite{hausdorff:19} gave a simple proof of
Tietze's extension theorem. For a nonempty closed set $A\subset X$ and
a continuous function $\varphi:A\to [0,1]$, he defined a continuous
extension $f:X\to \R$ of $\varphi$ by
\begin{equation}
  \label{eq:special-fn-Tietze-v3:1}
  f(p)= 
  \inf_{a\in A}\left[\varphi(a)+\frac{d(a,p)}{d(p,A)}-1\right],\quad
  \text{whenever $p\in X\setminus A$.}
\end{equation} 
In 1923, in his book \cite{kerekjarto:23}, Ker\'ekj\'art\'o refers to
a letter of Riesz which contains a simple proof of Tietze's extension
theorem, and presented this proof. The same proof was also presented
in the book of Alexandroff and Hopf \cite{61.0602.07} and credited
again to Riesz. In Riesz's construction, $\varphi:A\to [1,2]$ and the
extension $f:X\to \R$ is given by the following very simple formula:
\begin{equation}
  \label{eq:special-fn-Tietze-v3:2}
  f(p)=\sup_{a\in A} \varphi(a)\cdot\frac{d(p,A)}{d(a,p)},\quad
  \text{whenever $p\in X\setminus A$.}
\end{equation}
In his book \cite{MR0120319}, Dieudonn\'e included a proof of Tietze's
extension theorem which is virtually the same as previously given by
Riesz. In the same setting, he defined an extension $f:X\to \R$ of
$\varphi$ by
\begin{equation}
  \label{eq:special-fn-Tietze-v3:3}
  f(p)=
  \inf_{a\in A}\varphi(a)\cdot\frac{d(a,p)}{d(p,A)},\quad
  \text{whenever $p\in X\setminus A$.}
\end{equation}  

In 1951, Dugundji \cite{dugundji:51} generalised Theorem
\ref{theorem-special-fn-Tietze-v3:1} by replacing the range $\R$ with
an arbitrary locally convex topological vector space using virtually
the same method as Brouwer did. This was done at the cost of applying
A.~H.\ Stone's theorem \cite[Corollary 1]{stone:48} that each
metrizable space is paracompact. On the other hand, generalising a
result of Borsuk \cite{Borsuk1933} and Kakutani \cite{Kakutani1940},
Dugundji obtained the following very interesting application. For a
space $Z$, let $C^*(Z)$ be the Banach space of all bounded continuous
functions on $Z$ equipped with the sup-norm $\|\cdot\|$.

\begin{theorem}[Dugundji, 1951]
  \label{theorem-UC-Extensions-v9:1}
  Let $A\subset X$ be a closed set in a metric space $(X,d)$. Then
  there exists a linear map $\Phi:C^*(A)\to C^*(X)$ such that
  \[
    \Phi[\varphi]\uhr A=\varphi\quad \text{and}\quad
    \left\|\Phi[\varphi]\right\|=\|\varphi\|,\quad\text{for every
      $\varphi\in C^*(A)$.}
  \]
\end{theorem}

Regarding Theorem \ref{theorem-UC-Extensions-v9:1}, let us explicitly
remark that all extension constructions mentioned above are actually
``simultaneous'' extensions from some subset of $C^*(A)$ to $C^*(X)$,
see the next sections.\medskip

Extensions of uniformly continuous maps came naturally in the context
of the completion problem of metric spaces, see Hausdorff's 1914
monograph ``Grundz\"uge der Mengenlehre'' \cite{hausdorff:14}.  The
development of such extensions was gradual, the interested reader is
referred to \cite{Bentley_1998,Husek2010a} for an interesting outline
on the history of this extension problem. One of the first explicit
solutions of the extension problem for uniformly continuous functions
was obtained by McShane \cite[Corollary 2]{MR1562984}.

\begin{theorem}[McShane, 1934]
  \label{theorem-Uniform_cont-Ext-v1:1}
  Let $(X,d)$ be a metric space and ${A\subset X}$. Then each
  uniformly continuous bounded function $\varphi:A\to \R$ can be
  extended to the whole of $X$ preserving the uniform continuity and
  the bounds.
\end{theorem}

McShane obtained the above result as an application of his Lipschitz
extension theorem \cite[Theorem 1]{MR1562984}, see also Whitney
\cite[the footnote on p.\ 63]{MR1501735}. His construction was
implicit and based on moduli of continuity of bounded uniformly
continuous functions. Regarding the relationship with Lipschitz
extensions, let us explicitly remark that, recently, G. Beer
\cite{zbMATH07212343} refined McShane's construction to show that a
bounded uniformly continuous function on $A$ is actually Lipschitz
with respect to a metric on $X$ which is uniformly equivalent to
$d$.\medskip

In 1990, Mandelkern \cite{MR1076066} gave the following explicit
construction of the extension in Theorem
\ref{theorem-Uniform_cont-Ext-v1:1}, thus relating it to Tietze's
extension theorem.

\begin{theorem}[Mandelkern, 1990]
  \label{theorem-Mandelkern:1}
  Let $(X,d)$ be a metric space, $A\subset X$ be a nonempty closed set
  and $\varphi:A\to [1,2]$ be uniformly continuous. Then the extension
  $f:X\to \R$ defined as in \eqref{eq:special-fn-Tietze-v3:3} is also
  uniformly continuous.
\end{theorem}

In his review of Mandelkern's paper \cite{MR1076066}, St.\ Watson
\cite{Watson-1990-MR} remarked that Hausdorff's formula
\eqref{eq:special-fn-Tietze-v3:1} is simpler and even included in
Engelking's General Topology in \cite[Exercise 4.1.F]{engelking:89},
and questioned whether it also preserves uniform continuity. He
further remarked that the study of the mechanics of Tietze's extension
theorem is worthwhile and proposed the following general formula of
constructing an extension. He wrote that we can measure how close
$a\in A$ is to $p\in X\setminus A$ by the function
$s(a,p) = \frac{d(a,p)}{d(p, A)}$. Then he suggested to consider a
general function $c(s,r)$ representing the extension by
$f(p)=\inf_{a\in A}c(s(a,p),\varphi(a))$. For instance, in Hausdorff's
approach, $c(s,r)= r+s-1$. Thus, he posed the following general
question.

\begin{question}
  \label{question-UC-Extensions-v6:1}
  Can those functions $c(s,r)$ which preserve continuity and uniform
  continuity, respectively, be characterised?
\end{question}

We are now ready to state also the main purpose of this paper. In the
next section, we show that Hausdorff's extension construction
\eqref{eq:special-fn-Tietze-v3:1} preserves the uniformly continuous
functions (Theorem \ref{theorem-special-fn-Tietze-v8:1}). In fact, we
show that this construction defines an extension operator
$\Psi:C^*(A)\to C^*(X)$ which is an isometry and
$\Psi[\varphi]:X\to \R$ is uniformly continuous whenever
$\varphi\in C^*(A)$ is. Regarding this, let us explicitly remark that
$\Psi$ is not linear. In fact, as shown in the Remark after Corollary
D in Pe{\l}czy\'nski \cite[Notes and Remarks]{pelczynski:68}, see also
Lindenstrauss \cite{Lindenstrauss1964}, one cannot expect a linear
extension operator from $C^*(A)$ to $C^*(X)$ to preserve uniform
continuity. In Section \ref{sec:tietze-extenders}, we deal with an
alternative setting of Question~\ref{question-UC-Extensions-v6:1}. In
contrast to Watson's approach, we consider and abstract function
$\mathbf{F}(s,t)$ of the variables $s\geq t>0$, where $s$ plays the
role of $d(a,p)$ and $t$ --- that of $d(p,A)$, for $a\in A$ and
$p\in X\setminus A$. Then we associate the function
$\mathbf{F}^*:A\times (X\setminus A)\to \R$ defined by
$\mathbf{F}^*(a,p)=\mathbf{F}(d(a,p), d(p,A))$, $a\in A$ and
$p\in X\setminus A$. The extension of a bounded function
$\varphi:A\to [0,+\infty)$ is now defined by taking supremum or
infimum on $A$ of the product ``$\varphi(a)\cdot
\mathbf{F}^*(a,p)$''. By considering the multiplicative inverse of
$\mathbf{F}$, if necessary, the construction is reduced only to taking
supremum. In this case, we require $\mathbf{F}$ to take values in
$(0,1]$ and place three conditions on it, see
\eqref{eq:special-fn-Tietze-v6:1}, \eqref{eq:special-fn-Tietze-v6:2}
and \eqref{eq:special-fn-Tietze-v6:3}, which are satisfied by all
classical constructions of this type. In this general setting, we show
that
$\Omega_\mathbf{F}[\varphi](p)=\sup_{a\in A}\varphi(a)\cdot
\mathbf{F}^*(a,p)$, $p\in X\setminus A$, transforms each member
$\varphi\in C_+^*(A)$ of the positive cone of $C^*(A)$ into a member
$\Omega_\mathbf{F}[\varphi]\in C_+^*(X)$ which is uniformly continuous
whenever $\varphi$ is, Theorem \ref{theorem-special-fn-v2:1}. In fact,
this defines a sublinear extension operator
$\Omega_\mathbf{F}:C^*_+(A)\to C^*_+(X)$ which is an isotone
isometry. Based on a classical construction, we extend
$\Omega_\mathbf{F}$ to a positively homogeneous extension operator
$\Theta_\mathbf{F}:C^*(A)\to C^*(X)$ which is $2$-Lipschitz and still
preserves uniform continuity, Theorem
\ref{theorem-UC-Extensions-v22:1}.  Section \ref{sec:tietze-extenders}
also contains several supporting examples and remarks.  In the last
Section \ref{sec:bohrs-extens-constr}, we consider Bohr's extension
construction which is somewhat different being based on an integral
formula. This construction may look a bit artificial, but simple
arguments show that it also preserves the uniformly continuous
functions (Theorem
\ref{theorem-special-fn-Tietze-v11:1}). Furthermore, this construction
gives at once a sublinear extension operator $\Phi:C^*(A)\to C^*(X)$
which is an isotone isometry and preserves uniform
continuity. Finally, let us remark that the extension operators
constructed in this paper are complementary to a result for
simultaneous extension of bounded uniformly continuous function
obtained in \cite[Theorem 3.1]{Banakh2009}. The method used in
\cite{Banakh2009} is based on the original arguments of McShane for
proving Theorem \ref{theorem-Uniform_cont-Ext-v1:1}. Our extension
operators can be also compared with a nonlinear extension operator
constructed by Borsuk in \cite{borsuk:36}.

\section{Hausdorff's Extension Operator}
\label{sec:extens-constr-haud}

An \emph{ordered vector space} is a real vector space $E$ endowed with
a partial order $\leq$ such that for every $u, v \in E$ with
$u\leq v$,
\begin{enumerate}[label=(\roman*)]
\item $u+x\leq v+x$, for every $x\in E$,
\item $\alpha \cdot u\leq \alpha \cdot v$, for every $\alpha>0$.
\end{enumerate}
In such a space $E$, the set $\{v\in E: v\geq \mathbf{0}\}$ is called
the \emph{positive cone} of $E$, and denoted by $E^+$ or $E_+$.  An
ordered vector space $(E,\leq)$ is a \emph{Riesz space}, or a
\emph{vector lattice}, if it is also a lattice, i.e.\ if each pair of
elements $u,v\in E$ has a supremum $u\vee v=\sup\{u,v\}\in E$ and an
infimum $u\wedge v=\inf\{u,v\}\in E$. In a Riesz space $E$, to each
element $u\in E$ we may associate its absolute value $|u|=(-u)\vee
u$. A \emph{normed Riesz space} is a Riesz space $(E,\leq)$ endowed
with a norm $\|\cdot\|$ such that $\|u\|\leq \|v\|$, whenever
$u,v\in E$ with $|u|\leq |v|$.  A complete normed Riesz space is
called a \emph{Banach lattice}.\medskip

A map $\Phi:E\to V$ between Riesz spaces is \emph{positive} if
${\Phi(E^+) \subset V^+}$, and $\Phi$ is called \emph{isotone}, or
\emph{order-preserving}, if $\Phi(u)\leq \Phi(v)$ whenever $u,v\in E$
with $u\leq v$. Each linear positive operator between Riesz spaces is
isotone because for elements $u, v\in E$ we have that $u\leq v$
precisely when $v-u\in E^+$.\medskip

Let $C^*(Z)$ be the Banach space of all bounded, continuous,
real-valued functions on a space $Z$ equipped with the sup-norm
$\|f\|=\sup|f|$, $f\in C^*(Z)$. The vector space $C^*(Z)$ is
partially ordered by defining that $f\leq g$ holds whenever
$f(z) \leq g(z)$ for all $z\in Z$. This makes $C^*(Z)$ a Banach
lattice. Namely, for $f,g\in C^*(Z)$, the elements $f\vee g$ and
$f\wedge g$ are the pointwise maximum and minimum of $f$ and $g$. The
positive cone of $C^*(Z)$ will be denoted by $C^*_+(Z)$; and for
$t\in\R$, we will use $t_Z$ for the constant function $t_Z:Z\to\{t\}$.
In case $(Z,\rho)$ is a metric space, we will also use $C^*_u(Z)$ for
the uniformly continuous members of $C^*(Z)$.\medskip

In this section and the rest of the paper, $(X,d)$ is a fixed metric
space and $A\subset X$ is a nonempty closed set.  A map
$\Phi:C^*(A)\to C^*(X)$ is called an \emph{extension operator} if
$\Phi[\varphi]\uhr A=\varphi$, for every $\varphi\in C^*(A)$. A linear
map $\Phi:C^*(A)\to C^*(X)$ is called \emph{regular} if
$\Phi[1_A]=1_X$ and $\|\Phi[\varphi]\|=\|\varphi\|$,
$\varphi\in C^*(A)$, see \cite{pelczynski:68}. Evidently, a regular
linear map $\Phi:C^*(A)\to C^*(X)$ is an isometry. Furthermore,
according to \cite[Proposition 1.2]{pelczynski:68}, a linear operator
$\Phi:C^*(A)\to C^*(X)$ is regular if and only if $\Phi$ is positive
and $\Phi[1_A]=1_X$. In particular, each regular linear operator
$\Phi:C^*(A)\to C^*(X)$ is also isotone.\medskip

In what follows, we shall say that a map $\Psi:C^*(A)\to C^*(X)$
\emph{preserves uniform continuity} if
$\Psi[C^*_u(A)]\subset C^*_u(X)$. Also, for convenience, we shall say
that $\Psi$ \emph{preserves the constants} if $\Psi[t_A]=t_X$ for
every $t\in \R$. In this section, we consider Hausdorff's extension
construction \eqref{eq:special-fn-Tietze-v3:1} as an extension
operator from $C^*(A)$ to $C^*(X)$. While this operator is not linear,
it preserves uniform continuity and some of the properties of regular
linear operators.

\begin{theorem}
  \label{theorem-special-fn-Tietze-v8:1}
  For each $\varphi\in C^*(A)$, let $\Psi[\varphi]:X\to \R$ be the
  extension of $\varphi$ defined as in
  \eqref{eq:special-fn-Tietze-v3:1}, i.e.\ by
  \begin{equation}
    \label{eq:UC-Extensions-v22:1}
    \Psi[\varphi](p)= \inf_{a\in
      A}\left[\varphi(a)+\frac{d(a,p)}{d(p,A)}-1\right],\quad
    \text{whenever $p\in X\setminus A$.}
  \end{equation}
  Then $\Psi:C^*(A)\to C^*(X)$ is an extension operator which is an
  isotone isometry. Moreover, $\Psi$ preserves both uniform continuity
  and the constants.
\end{theorem}

The proof of Theorem \ref{theorem-special-fn-Tietze-v8:1} is based on
two general observations, the first of which is related to another
construction of Hausdorff about Lipschitz functions. Let us recall
that a map $f:X\to Y$ in a metric space $(Y,\rho)$ is \emph{Lipschitz}
if there exists $K\geq 0$ such that $\rho(f(p),f(q))\leq K d(p,q)$,
for all $p,q\in X$.  In this case, to emphasise on the constant $K$,
we also say that $f$ is \emph{$K$-Lipschitz}.\medskip

In his 1919 paper \cite{hausdorff:19}, Hausdorff described a very
interesting construction of Lipschitz functions. Namely, to each
bounded function $\varphi:X\to \R$ and $\kappa>0$, he associated the
function $f_\kappa:X\to \R$ defined by
\[
  f_\kappa(p)=\inf_{x\in X}\left[\varphi(x)+ \kappa
    d(x,p)\right],\quad p\in X.
\]
He showed that $f_\kappa$ is $\kappa$-Lipschitz, see \cite[page
293]{hausdorff:19}, and credited the construction to Moritz Pasch. In
the literature, the functions $f_\kappa$, $\kappa>0$, are often called
the \emph{Pasch-Hausdorff envelope} of $\varphi$; the function
$f_\kappa$ is also called the \emph{$\kappa$-Lipschitz regularisation}
of $\varphi$. It should be remarked that $f_\kappa$ can be defined for
each function $\varphi:X\to \R$ for which there exists a
$\kappa$-Lipschitz function $f:X\to \R$ with $f\leq \varphi$. In fact,
one can easily see that $f_\kappa$ is the greatest $\kappa$-Lipschitz
function with $f_\kappa \leq \varphi$.\medskip

For an extended discussion on the Pasch-Hausdorff construction, the
interested reader is referred to \cite{Gutev2020}. In the proof of
Theorem \ref{theorem-special-fn-Tietze-v8:1}, we will use the
following slight modification of this construction.

\begin{proposition}
  \label{proposition-UC-Extensions-v19:1}
  For a bounded function $\varphi:A\to [0,+\infty)$, define a function
  $f:X\to \R$ by
  \begin{equation}
    \label{eq:UC-Extensions-v19:1}
    f(p)=\inf_{a\in A}\big[\varphi(a)
    d(p,A) + d(a,p)\big],\quad p\in X.
  \end{equation}
  Then $f$ is Lipschitz.
\end{proposition}

\begin{proof}
  Suppose that $\varphi:A\to [0,\kappa]$ for some $\kappa>0$.  If
  $x,p\in X$ and $a\in A$, then $d(x,A)\leq d(p,A)+ d(x,p)$ and
  $d(a,x)\leq d(a,p)+ d(x,p)$. Multiplying the first inequality with
  $\varphi(a)$ and adding it to the second one, it follows from
  \eqref{eq:UC-Extensions-v19:1} that
  $f(x)\leq f(p)+ (\kappa+1)d(x,p)$ because $\varphi$ is bounded by
  $\kappa$. Accordingly, $|f(x)-f(p)|\leq (\kappa+1)d(x,p)$.
\end{proof}

For the other general observation, let us briefly review the
difference between continuity ``at a point'' and continuity ``on a
subset''. In some sources, $f:X\to \R$ is continuous on $A$ if it is
continuous at each point of $A$. In other sources, $f$ is continuous
on $A$ if its restriction $f\uhr A:A\to \R$ is continuous. To avoid
any misunderstanding, we will say that $f$ is \emph{continuous at the
  points of $A$} if it is continuous at each point of $A$. Similarly,
in most sources, $f$ is assumed to be uniformly continuous on $A$ if
its restriction $f\uhr A$ is uniformly continuous. The provision for
the other interpretation was made in \cite{Beer1991} where $f$ was
called ``uniformly continuous on $A$'' if for each $\varepsilon > 0$
there exists $\delta > 0$ such that $|f(x)- f(a)| < \varepsilon$, for
every $a\in A$ and $x\in X$ with $d(a,x) < \delta$. The terminology
was further refined in \cite[Definition 1.1]{Beer2009}, where this
property was called \emph{strongly uniformly continuous on $A$}. To
avoid any confusion with the existing literature, we will retain this
concept. However, for the proper understanding of our results, let us
explicitly remark that $f$ is strongly uniformly continuous on $A$
pre\-ci\-sely when it is continuous at the points of $A$ with the same
$\delta$ for all points of $A$.\medskip

If a function $f:X\to \R$ is continuous at the points of the closed
set $A$ and its restriction on $X\setminus A$ is also continuous, then
$f$ is itself continuous. This follows from the fact that
$X\setminus A$ is open, hence $f\uhr X\setminus A$ is continuous
precisely when $f$ is continuous at the points of $X\setminus A$. In
the case of strong uniform continuity, this property was summarised by
M. Hu\v{s}ek \cite[Corollary 16]{Husek2010a}. It is in good accord
with our method of proving Theorem
\ref{theorem-special-fn-Tietze-v8:1}, and we briefly reproduce the
arguments. To this end, for $\varepsilon>0$ we will use
$\mathbf{O}(p,\varepsilon)=\{x\in X: d(x,p)<\varepsilon\}$ to denote
the \emph{open $\varepsilon$-ball} centred at a point $p\in X$. Also,
let
$\mathbf{O}(S,\varepsilon) = \bigcup_{p\in
  S}\mathbf{O}(p,\varepsilon)$ be the $\varepsilon$-\emph{enlargement}
of $S$, whenever $S\subset X$.

\begin{proposition}
  \label{proposition-Mandelkern:1}
  If a function $f:X\to \R$ is strongly uniformly continuous on $A$
  and for each $\tau>0$, its restriction on
  $X\setminus \mathbf{O}(A,\tau)$ is uniformly continuous, then $f$ is
  itself uniformly continuous.
\end{proposition}

\begin{proof}
  Let $\varepsilon>0$. Since $f$ is strongly uniformly continuous on
  $A$, there exists $\tau>0$ such that $|f(x)-f(a)|<\varepsilon$ for
  every $a\in A$ and $x\in X$ with $d(a,x)<2\tau$. For the same
  reason, since the restriction of $f$ on
  $X\setminus \mathbf{O}(A,\tau)$ is uniformly continuous, there
  exists $0<\delta\leq\tau$ such that $|f(x)-f(p)|<\varepsilon$ for
  every $x,p\notin \mathbf{O}(A,\tau)$ with $d(x,p)<\delta$. To show
  that this $\delta$ works, take points $x,p\in X$ with
  $d(x,p)<\delta$. If $x,p\notin \mathbf{O}(A,\tau)$, then
  $|f(x)-f(p)|<\varepsilon$. If $d(x,A)<\tau$ or $d(p,A)<\tau$, there
  exists a point $a\in A$ with $d(x,a)<2\tau$ and $d(p,a)<2\tau$
  because $d(x,p)<\delta\leq \tau$. Accordingly,
  $|f(x)-f(p)|\leq |f(x)- f(a)| +|f(a)-f(p)|<2\varepsilon$.
\end{proof}

In the proof of Theorem \ref{theorem-special-fn-Tietze-v8:1} and what
follows, for convenience, we set
\begin{equation}
  \label{eq:UC-Extensions-v6:1}
  \mathbf{O}_A(p,\varepsilon)=\mathbf{O}(p,\varepsilon)\cap A,\quad
  \text{for every $p\in X$ and $\varepsilon>0$}. 
\end{equation}
Evidently, the set $\mathbf{O}_A(p,\varepsilon)$ could be empty for
some $p\in X\setminus A$. However,
$\mathbf{O}_A(p,\varepsilon)\neq \emptyset$ for every
$\varepsilon>d(p,A)$. In fact,
$d(p,A)=\inf_{a\in \mathbf{O}_A(p,\varepsilon)} d(p,a)$, whenever
$\varepsilon>d(p,A)$. Below, we translate this property only in terms
of the open balls centred at the points of $A$. Namely, if $p\in A$
and $\tau>0$, then
\begin{equation}
  \label{eq:UC-Extensions-v20:1}
  d(x,A)={\inf}\big\{d(x,a): a\in \mathbf{O}_A(p,2\tau)\big\},
  \quad \text{for every $x\in \mathbf{O}(p,\tau)$.}
\end{equation}
This follows from the fact that
$\emptyset\neq \mathbf{O}_A(x,\tau)\subset
\mathbf{O}_A(p,2\tau)$.\medskip

Finally, we will also use the property of bounded functions
$g,h:Z\to \R$ that
\begin{equation}
  \label{eq:special-fn-Tietze-v9:2}
  | \inf g - \inf h|\leq \sup |g -
  h|\quad\text{and}\quad |\sup g-\sup h|\leq \sup |g - h|.  
\end{equation}
It is a simple consequence of the fact that
$\sup(B+C)=\sup B+ \sup C$, whenever $B,C\subset \R$ are nonempty
bounded sets.

\begin{proof}[Proof of Theorem \ref{theorem-special-fn-Tietze-v8:1}]
  Let $\varphi:A\to \R$ be a bounded (uniformly) continuous function
  and $\Psi[\varphi]:X\to \R$ be the extension of $\varphi$ defined as
  in \eqref{eq:UC-Extensions-v22:1}. Whenever $\lambda\in \R$, it is
  evident from \eqref{eq:UC-Extensions-v22:1} that
  $\Psi[\varphi+\lambda]=\Psi[\varphi]+\lambda$. Accordingly, we may
  assume that $\varphi:A\to [0,\kappa]$ for some $\kappa\geq 1$. In
  this setting, we will first show that $\Psi[\varphi]$ is continuous
  at the points of $A$.  Also, that it is strongly uniformly
  continuous on $A$ provided $\varphi$ is uniformly continuous. So,
  take a point $p\in A$ and $0<\varepsilon\leq\kappa$. Since
  $\varphi:A\to [0,\kappa]$ is continuous at $p$, there exists
  $\delta>0$ such that $|\varphi(a)-\varphi(p)|<\varepsilon$ for every
  $a\in \mathbf{O}_A(p,4\kappa\delta)$. If $\varphi$ is uniformly
  continuous, we may assume that this $\delta$ is the same for all
  points of $A$, namely that $|\varphi(a)-\varphi(b)|<\varepsilon$,
  for every $a,b\in A$ with $d(a,b)<4\kappa\delta$.  Take
  $x\in \mathbf{O}(p,\delta)\setminus A$. If
  $a\in \mathbf{O}_A(p,4\kappa\delta)$, then
  $ \varphi(p)-\varepsilon<\varphi(a)
  +\frac{d(a,x)}{d(x,A)}-1<[\varphi(p)+\varepsilon]+
  \frac{d(a,x)}{d(x,A)}-1$ because $\frac{d(a,x)}{d(x,A)}-1\geq 0$.
  Since $\varphi(p)\leq \kappa$ and $\varepsilon\leq\kappa$, it
  follows from \eqref{eq:UC-Extensions-v20:1} that
  \[
    \varphi(p)-\varepsilon\leq \inf_{a\in
      \mathbf{O}_A(p,4\kappa\delta)}\left[
      \varphi(a)+\frac{d(a,x)}{d(x,A)}-1\right] \leq
    \varphi(p)+\varepsilon\leq2\kappa.
  \]
  If $a\in A\setminus \mathbf{O}(p,4\kappa\delta)$, then
  $\varphi(a)+\frac{d(a,x)}{d(x,A)}-1> \varphi(a)+3\kappa-1\geq
  2\kappa$ because $\varphi(a)\geq 0$, $\kappa\geq 1$ and
  $d(a,x)\geq 3\kappa\delta>3\kappa d(p,x)\geq 3\kappa
  d(x,A)$. Accordingly,
  $\Psi[\varphi](x)=\inf_{a\in \mathbf{O}_A(p,4\kappa\delta)}\left[
    \varphi(a)+\frac{d(a,x)}{d(x,A)}-1\right]$ and, therefore,
  $|\Psi[\varphi](x)-\varphi(p)|\leq \varepsilon$.\smallskip

  Next, take $\tau>0$, and let us show that
  $\Psi[\varphi]\uhr X\setminus \mathbf{O}(A,\tau)$ is uniformly
  continuous.  To this end, let $f:X\to \R$ be the function defined as
  in \eqref{eq:UC-Extensions-v19:1} with respect to the given function
  $\varphi:X\to [0,\kappa]$. Then by Proposition
  \ref{proposition-UC-Extensions-v19:1}, $f$ is Lipschitz.  Since
  $\Psi[\varphi](x)=\frac{f(x)}{d(x,A)}-1$, $x\in X\setminus A$, and
  the restriction of $\frac1{d(\cdot,A)}$ on
  $X\setminus \mathbf{O}(A,\tau)$ is Lipschitz,
  $\Psi[\varphi]\uhr X\setminus \mathbf{O}(A,\tau)$ is also
  Lipschitz. Thus, $\Psi[\varphi]$ is continuous and by Proposition
  \ref{proposition-Mandelkern:1}, it is also uniformly continuous
  provided so is $\varphi$. In other words, $\Psi:C^*(A)\to C^*(X)$ is
  an extension operator which preserves uniform continuity. Moreover,
  by \eqref{eq:UC-Extensions-v22:1}, $\Psi$ is isotone and preserves
  the constants as well. Finally, take $\varphi,\psi\in C^*(A)$. Then
  for $p\in X\setminus A$, it follows from
  \eqref{eq:UC-Extensions-v22:1} and \eqref{eq:special-fn-Tietze-v9:2}
  that
  \[
    |\Psi[\varphi](p)-\Psi[\psi](p)|\leq \sup_{a\in
      A}|\varphi(a)-\psi(a)|=\|\varphi-\psi\|.
  \]
  Hence, $\|\Psi[\varphi]-\Psi[\psi]\|=\|\varphi-\psi\|$ because
  $\Psi$ is an extension operator.
\end{proof}

We conclude this section with two remarks regarding Theorem
\ref{theorem-special-fn-Tietze-v8:1}.

\begin{remark}
  \label{remark-UC-Extensions-v22:1}
  The extension operator $\Psi:C^*(A)\to C^*(X)$ defined in
  \eqref{eq:UC-Extensions-v22:1} is not necessarily linear. For
  instance, take $X=[-1,+\infty)$ and $A=[-1,0]$. Also, let
  $\varphi(a)=a$, $a\in A$, be the identity function.  If $a\in A$ and
  $p\in X\setminus A$, then $d(a,p)=p-a$ and $d(p,A)=p$. Hence, for
  $k\in \N$ and $p\in X\setminus A$, we get that
  $$
  \Psi[k\varphi](p)=\inf_{a\in
    A}\left[ka+\frac{p-a}p-1\right]=\inf_{-1\leq a\leq
    0}a\left[k-\frac1p\right]=
  \begin{cases}
    0 &\text{if $p\leq \frac1k$,}\\
    \frac1p-k &\text{if $p\geq \frac1k$.}
  \end{cases}
  $$
  Accordingly,
  $\Psi[2\varphi]=\Psi[\varphi+\varphi]\neq 2\Psi[\varphi]$.\qed
\end{remark}

\begin{remark}
  \label{remark-UC-Extensions-v22:2}
  For a nonempty set $Z$, let $\ell_\infty(Z)$ be the Banach lattice
  of all bounded real-valued functions on $Z$ equipped with the
  sup-norm ${\|\varphi\|=\sup |\varphi|}$,
  $\varphi\in \ell_\infty(Z)$. The partial order on $\ell_\infty(Z)$
  is defined precisely as that for $C^*(Z)$, in fact $C^*(Z)$ is a
  Banach sublattice of $\ell_\infty(Z)$. The positive cone of
  $\ell_\infty(Z)$ will be denoted by $\ell_\infty^+(Z)$. Evidently,
  in our setting of $(X,d)$ and $A\subset X$,
  \eqref{eq:UC-Extensions-v22:1} defines an extension operator
  $\Psi:\ell_\infty(A)\to \ell_\infty(X)$. Moreover, the same argument
  as in Theorem \ref{theorem-special-fn-Tietze-v8:1} shows that
  $\Psi:\ell_\infty(A)\to \ell_\infty(X)$ is an isotone isometry and,
  in particular, $\Psi[\ell_\infty^+(A)]\subset \ell_\infty^+(X)$. In
  this interpretation, $\Psi$ preserves \emph{continuity} in the sense
  that $\Psi[C^*(A)]\subset C^*(X)$.\qed
\end{remark}

\section{Tietze-Like Extension Operators}
\label{sec:tietze-extenders}

If $a\in A$ and $p\in X\setminus A$, then $s=d(a,p)\geq
d(p,A)=t>0$. Motivated by this, we consider the set
\[
  \Delta=\left\{(s,t)\in\R^2: s\geq t>0\right\}.
\]
Next, to each $\mathbf{F}:\Delta\to \R$ we will associate the
``composite'' function
\[
  \mathbf{F}^*:A\times(X\setminus A)\ni (a,p)\longrightarrow
  \big(d(a,p),d(p,A)\big)\in \Delta \xrightarrow[]{\mathbf{F}}\R.
\]
In other words, $\mathbf{F}^*:A\times(X\setminus A)\to \R$ is defined
by $\mathbf{F}^*(a,p)=\mathbf{F}\big(d(a,p),d(p,A)\big)$, for every
$a\in A$ and $p\in X\setminus A$.  Finally, following Tietze's
construction \eqref{eq:special-fn-Tietze:1}, for
$\mathbf{F}:\Delta\to (0,1]$ and a bounded function
$\varphi:A\to [0,+\infty)$, let
$\Omega_\mathbf{F}[\varphi]:X\to [0,+\infty)$ be the extension of
$\varphi$ defined by
\begin{equation}
  \label{eq:special-fn-Tietze-v5:1}
  \Omega_\mathbf{F}[\varphi](p)=
  \sup_{a\in A} \varphi(a)\cdot \mathbf{F}^*(a,p),\quad
  \text{for every $p\in X\setminus A$.} 
\end{equation}
Since $0<\mathbf{F}\leq 1$, it follows that
$\varphi(a)\cdot \mathbf{F}^*(a,p)\leq \varphi(a)$, for every $a\in A$
and ${p\in X\setminus A}$. Accordingly, $\Omega_\mathbf{F}[\varphi]$
is also a bounded function, in fact it is bounded above by
$\|\varphi\| = \sup_{a\in A}\varphi(a)$. Thus,
$\Omega_\mathbf{F}:\ell_\infty^+(A)\to \ell_\infty^+(X)$ is an
extension operator which is an isotone isometry; the latter follows
easily from \eqref{eq:special-fn-Tietze-v9:2} and
\eqref{eq:special-fn-Tietze-v5:1}. Furthermore, $\Omega_\mathbf{F}$ is
\emph{sublinear}, i.e.\ for every ${\varphi,\psi\in \ell_\infty^+(A)}$
and $\lambda\geq 0$ we have that
$\Omega_\mathbf{F}[\varphi+\psi]\leq \Omega_\mathbf{F}[\varphi] +
\Omega_\mathbf{F}[\psi]$ ($\Omega_\mathbf{F}$ is \emph{subadditive})
and
$\Omega_\mathbf{F}[\lambda\varphi] = \lambda
\Omega_\mathbf{F}[\varphi]$ ($\Omega_\mathbf{F}$ is \emph{positively
  homogeneous}).\medskip

In this section, we will show that the extension operator
$\Omega_\mathbf{F}:\ell_\infty^+(A)\to \ell_\infty^+(X)$ preserves
both continuity and uniform continuity based only on the following
three general properties of $\mathbf{F}:\Delta\to (0,1]$.

\begin{gather}
  \label{eq:special-fn-Tietze-v6:1}
  \begin{cases}
    \lim_{t\to 0}\mathbf{F}(t,t)=1, &\text{and}\\
    \lim_{t\to 0}\mathbf{F}(s,t)=0, &\text{for every fixed $s>0$;}
  \end{cases} \\
  \label{eq:special-fn-Tietze-v6:2}
  \mathbf{F}(\cdot,t)\ \text{is
    decreasing, for every fixed $t>0$;}\\
  \label{eq:special-fn-Tietze-v6:3}
  \text{Whenever $\tau>0$,}\ \mathbf{F}(s,t)\ \text{is uniformly
    continuous on $s\geq t\geq \tau$.}
\end{gather}
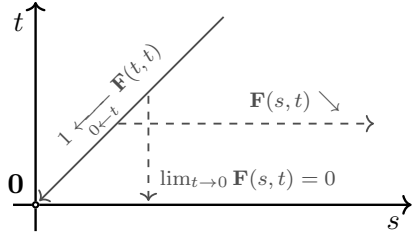
\begin{center}
  \begin{tikzpicture}[baseline,thick]
    \draw[line width=.9pt, ->] (-.3,0) -- (5,0) node[anchor=north
    east] {\small $s$}; \draw[line width=.9pt,->] (0,-.35) -- (0,2.7)
    node[anchor=north east] {\small $t$}; \draw[white!30!black,line
    width=0.7pt, <-] (0.05,0.05) -- (2.5,2.5); \draw[white!30!black]
    (1.3,1.15) node[rotate=45,anchor=south] {\tiny
      $1\xleftarrow[0\leftarrow t]{}\mathbf{F}(t,t)$};
    \draw[white!30!black,line width=0.7pt, dashed,->] (1.5,1.45) --
    (1.5,0.05) node[anchor=south west] {\tiny
      $\lim_{t\to 0}\mathbf{F}(s,t)= 0$}; \draw[white!30!black,line
    width=0.7pt, dashed, ->] (1.1,1.08) -- (4.5,1.08)
    node[anchor=south east] {\tiny $\mathbf{F}(s,t)\searrow\quad$};
    \filldraw[white] (0,0) circle [radius=1pt]; \draw (0,0) circle
    [radius=1pt] node[anchor=south east] {\small $\mathbf{0}$};
  \end{tikzpicture}
\end{center}
In this setting, each function $\mathbf{F}:\Delta\to (0,1]$ satisfying
\eqref{eq:special-fn-Tietze-v6:1}, \eqref{eq:special-fn-Tietze-v6:2}
and \eqref{eq:special-fn-Tietze-v6:3} will be called a \emph{Tietze
  Extender}.

\begin{theorem}
  \label{theorem-special-fn-v2:1}
  Let $\mathbf{F}:\Delta\to (0,1]$ be a Tietze extender. Then the
  extension operator
  $\Omega_\mathbf{F}:\ell_\infty^+(A)\to \ell_\infty^+(X)$, defined as
  in \eqref{eq:special-fn-Tietze-v5:1}, is a sublinear isotone
  isometry which preserves both continuity and uniform continuity.
\end{theorem}

Regarding the proper place of Theorem \ref{theorem-special-fn-v2:1},
here are the simplest examples of two different types of Tietze
extenders.

\begin{example}
  \label{example-UC-Extensions-v6:1}
  Following \eqref{eq:special-fn-Tietze-v5:1} and Tietze's extension
  construction \eqref{eq:special-fn-Tietze:1}, let
  \begin{equation}
    \label{eq:special-fn-Tietze:2}
    \mathbf{T}(s,t)= \frac1{\big(1+s^2\big)^{\frac1t}},\quad
    s\geq t>0.
  \end{equation}    
  Then $\mathbf{T}:\Delta\to (0,1]$ is a Tietze extender. Indeed,
  \eqref{eq:special-fn-Tietze-v6:1} and
  \eqref{eq:special-fn-Tietze-v6:2} are evident for this function.  To
  see that $\mathbf{T}(s,t)$ also has the property in
  \eqref{eq:special-fn-Tietze-v6:3}, consider the function
  $\mathbf{G}(x,y)=\frac1{\left(1+x^2\right)^y}=e^{-y\ln\left(1+x^2\right)}$
  for $x\geq0$ and $y\geq0$. Then it is uniformly continuous being a
  continuous function with
  $\lim_{\substack{x\to+\infty\\ y\to+\infty}} \mathbf{G}(x,y)=
  0$. Accordingly, so is
  $\mathbf{T}(s,t)=\mathbf{G}\left(s,\frac1t\right)$ on the domain
  $s\geq t\geq \tau>0$. Thus, Theorem \ref{theorem-special-fn-v2:1} is
  valid for $\mathbf{T}(s,t)$. Since Tietze's extension construction
  \eqref{eq:special-fn-Tietze:1} is identical to that in
  \eqref{eq:special-fn-Tietze-v5:1}, it preserves not only continuity,
  but also uniform continuity. On the other hand, this construction
  doesn't preserves the constants (in particular, the bounds). For
  instance, $1_A:A\to [1,\mu]$ for every $\mu\geq 1$, but the
  extension $\Omega_\mathbf{T}[1_A]$ is not necessarily the constant
  function $1_X$ of $X$. Namely, $p\in X\setminus A$ implies that
  \[
    \Omega_\mathbf{T}[1_A](p)=\sup_{a\in A}\frac1{\left(1 +
        \left[d(a,p)\right]^2\right)^{\frac1{d(p,A)}}}= \frac1{\left(1
        + \left[d(p,A)\right]^2\right)^{\frac1{d(p,A)}}}<1.\hfill\qed
  \]
\end{example}

\begin{example}
  \label{example-UC-Extensions-v7:1}
  Following \eqref{eq:special-fn-Tietze-v5:1} and Riesz's extension
  construction \eqref{eq:special-fn-Tietze-v3:2}, let
  \begin{equation}
    \label{eq:special-fn-Tietze-v3:4}
    \mathbf{R}(s,t)= \frac ts,\quad s\geq t>0.
  \end{equation}
  It is obvious that $\mathbf{R}:\Delta\to (0,1]$ is a Tietze
  extender. Hence, Theorem \ref{theorem-special-fn-v2:1} is also valid
  for the Riesz function. In contrast to the Tietze function
  $\mathbf{T}(s,t)$ in \eqref{eq:special-fn-Tietze:2} of Example
  \ref{example-UC-Extensions-v6:1}, the Riesz function
  $\mathbf{R}(s,t)$ preserves the bounds because $\mathbf{R}(t,t)=1$
  for every $t>0$. Indeed, for $\varphi:A\to [\lambda,\mu]$ and
  $p\in X\setminus A$, where $\mu\geq \lambda\geq 0$, we have that
  \[
    \Omega_\mathbf{R}[\varphi](p)=\sup_{a\in A}\varphi(a)\cdot
    \mathbf{R}^*(a,p) \geq \lambda\sup_{a\in
      A}\frac{d(p,A)}{d(a,p)}=\lambda\frac{d(p,A)}{d(p,A)}=\lambda.
  \]
  Accordingly, Theorem \ref{theorem-special-fn-v2:1} gives McShane's
  result on uniformly continuous extensions at once. Let us remark
  that in Theorem \ref{theorem-Uniform_cont-Ext-v1:1}, McShane first
  extended $\varphi$ to a uniformly continuous function on $X$, and
  next adjusted its bounds.\qed
\end{example}

Following Dieudonn\'e's extension construction
\eqref{eq:special-fn-Tietze-v3:3}, for
$\mathbf{G}:\Delta\to [1,+\infty)$ and a bounded function
$\varphi:A\to [0,+\infty)$, we may define an alternative extension
$\mho_\mathbf{G}[\varphi]:X\to [\inf\varphi,+\infty)$ of $\varphi$ by
\begin{equation}
  \label{eq:UC-Extensions-v11:1}
  \mho_\mathbf{G}[\varphi](p)=\inf_{a\in A}\varphi(a)\cdot
  \mathbf{G}^*(a,p),\quad p\in X\setminus A.
\end{equation}
Evidently if $\mathbf{F}:\Delta\to (0,1]$, then
$\mathbf{G}=\frac1{\mathbf{F}}:\Delta\to [1,+\infty)$. Thus, for a
bounded function $\varphi:A\to (0,+\infty)$ we have that
\begin{equation}
  \label{eq:UC-Extensions-v11:2}
  \mho_{\frac1{\mathbf{F}}}[\varphi]=
  \frac1{\Omega_\mathbf{F}\left[\frac1\varphi\right]}. 
\end{equation}
  
\begin{example}
  \label{example-UC-Extensions-v6:2}
  Following Dieudonn\'e's extension construction
  \eqref{eq:special-fn-Tietze-v3:3}, define a function
  $\mathbf{D}:\Delta\to [1,+\infty)$ by
  \begin{equation}
    \label{eq:special-fn-Tietze-v3:5}
    \mathbf{D}(s,t)=\frac st,\quad s\geq t>0.
  \end{equation}
  Then for $\mu>\lambda>0$ and $\varphi:A\to [\lambda,\mu]$,
  Dieudonn\'e's extension is identical to that in
  \eqref{eq:UC-Extensions-v11:1}, and is given by
  $\mho_\mathbf{D}[\varphi]:X\to [\lambda,\mu]$. However, the
  multiplicative inverse $\frac1{\mathbf{D}}$ is the Riesz function
  $\mathbf{R}$ in \eqref{eq:special-fn-Tietze-v3:4}. Hence, by
  \eqref{eq:UC-Extensions-v11:2}, Example
  \ref{example-UC-Extensions-v7:1} and Theorem
  \ref{theorem-special-fn-v2:1},
  $\mho_\mathbf{D}[\varphi] =\frac1
  {\Omega_{\mathbf{R}}\left[\frac1\varphi \right]}:X\to
  \left[\lambda,\mu\right]$ is (uniformly) continuous whenever so is
  $\frac1\varphi$, equivalently $\varphi$. Accordingly, Dieudonn\'e's
  extension $\mho_\mathbf{D}[\varphi]$ is (uniformly) continuous
  precisely when so is $\varphi$. Thus, Theorem
  \ref{theorem-special-fn-v2:1} also contains Theorem
  \ref{theorem-Mandelkern:1} as a special case. \qed
\end{example}

The proof of Theorem \ref{theorem-special-fn-v2:1} is based on three
properties of Tietze extenders. The first one is the following relaxed
representation of the extension $\Omega_\mathbf{F}[1_A]$ of the
constant function $1_A$.

\begin{proposition}
  \label{proposition-special-fn-Tietze-v5:1}
  Let $\mathbf{F}:\Delta\to (0,1]$ be a continuous function as in
  \eqref{eq:special-fn-Tietze-v6:2}, and
  $\Omega_\mathbf{F}[1_A]:X\to (0,1]$ be the extension of\/
  $1_A:A\to \{1\}$ defined as in \eqref{eq:special-fn-Tietze-v5:1}.
  If $p\in A$, $\tau>0$ and $x\in \mathbf{O}(p,\tau)\setminus A$, then
  \[
    \Omega_\mathbf{F}[1_A](x)= \mathbf{F}(d(x,A),d(x,A))=\sup_{a\in
      \mathbf{O}_A(p,2\tau)} \mathbf{F}^*(a,x).
  \]
\end{proposition}

\begin{proof}
  Since $\mathbf{F}$ is continuous, it follows from
  \eqref{eq:UC-Extensions-v20:1} and \eqref{eq:special-fn-Tietze-v6:2}
  that
  \begin{align*}
    \Omega_\mathbf{F}[1_A](x)
    &=\mathbf{F}\left(\inf_{a\in A}d(a,x),d(x,A)\right)
      = \mathbf{F}(d(x,A),d(x,A))\\
    &=\mathbf{F}\left(\inf_{a\in \mathbf{O}_A(p,2\tau)}d(a,x),d(x,A)\right)
      =\sup_{a\in \mathbf{O}_A(p,2\tau)}
      \mathbf{F}^*(a,x).\qedhere
  \end{align*}
\end{proof}

Tietze extenders also have the following property related to uniform
continuity.

\begin{proposition}
  \label{proposition-special-fn-Tietze-v6:2}
  Let $\mathbf{F}:\Delta\to \R$ be a function satisfying
  \eqref{eq:special-fn-Tietze-v6:3}, and ${\varepsilon,\tau>0}$. Then
  there exists $\delta>0$ such that for every
  $x,p\in X\setminus\mathbf{O}(A,\tau)$ with $d(x,p)<\delta$,
  \begin{equation}
    \label{eq:special-fn-Tietze-v9:1}
    |\mathbf{F}^*(a,x)-\mathbf{F}^*(a,p)|<\varepsilon,\quad
    \text{for all $a\in A$.}
  \end{equation}
\end{proposition}

\begin{proof}
  By \eqref{eq:special-fn-Tietze-v6:3}, there is $\delta>0$ such that
  $|\mathbf{F}(s_1,t_1)-\mathbf{F}(s_0,t_0)|<\varepsilon$, for every
  $s_i\geq t_i\geq \tau$, $i=0,1$, with $|s_1-s_0|<\delta$ and
  $|t_1-t_0|<\delta$.  If $a\in A$ and
  $x,p\in X\setminus \mathbf{O}(A,\tau)$ with $d(x,p)<\delta$, then
  $s_0=d(a,p)\geq d(p,A)=t_0\geq \tau$ and
  ${s_1=d(a,x)\geq d(x,A)=t_1\geq\tau}$. Moreover,
  $|d(a,x)-d(a,p)|\leq d(x,p)<\delta$ and
  $|d(x,A)-d(p,A)|\leq d(x,p)<\delta$. 
  Hence, $|\mathbf{F}^*(a,x)-\mathbf{F}^*(a,p)|<\varepsilon$.
\end{proof}

We conclude the preparation for the proof of Theorem
\ref{theorem-special-fn-v2:1} with a crucial observation about
continuity and uniform continuity of the extension defined in
\eqref{eq:special-fn-Tietze-v5:1}.

\begin{proposition}
  \label{proposition-UC-Extensions-rv2:1}
  Let $\mathbf{F}:\Delta\to (0,1]$ be a continuous function satisfying
  \eqref{eq:special-fn-Tietze-v6:1} and
  \eqref{eq:special-fn-Tietze-v6:2}. If $\varphi\in \ell_\infty^+(A)$
  is a continuous function, then the extension
  $\Omega_\mathbf{F}[\varphi]\in\ell_\infty^+(X)$ defined as in
  \eqref{eq:special-fn-Tietze-v5:1} is continuous at the points of
  $A$. Moreover, $\Omega_\mathbf{F}[\varphi]$ is strongly uniformly
  continuous on $A$ provided $\varphi$ is uniformly continuous.
\end{proposition}

\begin{proof}
  Take a point $p\in A$ and $\varepsilon>0$. Since
  $\varphi:A\to [0,+\infty)$ is continuous at $p$, there exists
  $\tau>0$ such that $|\varphi(a)-\varphi(p)|<\varepsilon$ for every
  $a\in \mathbf{O}_A(p,2\tau)$, see \eqref{eq:UC-Extensions-v6:1}. If
  $\varphi$ is uniformly continuous, we may assume that this $\tau$ is
  the same for all points of $A$, namely that
  $|\varphi(a)-\varphi(b)|<\varepsilon$, for every $a,b\in A$ with
  $d(a,b)<2\tau$. Thus, it will be now sufficient to show that there
  exists $\delta>0$ depending only on $\tau$ and the function
  $\mathbf{F}$ such that
  \begin{equation}
    \label{eq:special-fn-Tietze-v4:2}
    \varphi(p)-2\varepsilon \leq \Omega_\mathbf{F}[\varphi](x)\leq
    \varphi(p)+\varepsilon,\quad \text{for every $x\in 
      \mathbf{O}(p,\delta)\setminus A$.}
  \end{equation}
  In the construction below, we treat each part of this inequality
  separately.\smallskip
  
  For convenience take $\lambda>\varepsilon$ such that
  $\varphi:A\to [0,\lambda]$. Then for $s\geq \tau\geq t>0$, it
  follows from \eqref{eq:special-fn-Tietze-v6:1} and
  \eqref{eq:special-fn-Tietze-v6:2} that
  $\mathbf{F}(s,t)\leq \mathbf{F}(\tau,t)\xrightarrow[t\to
  0]{}0<\frac\varepsilon\lambda$.  Hence, there is $0<r\leq \tau$ such
  that $\mathbf{F}(s,t)<\frac\varepsilon\lambda$ for every
  $0<t<r\leq\tau \leq s$. This implies the right-hand side of
  \eqref{eq:special-fn-Tietze-v4:2} for every $0<\delta\leq
  r$. Indeed, take $x\in \mathbf{O}(p,r)\setminus A$ and $a\in A$.  If
  $d(a,p)\geq 2\tau$, then $t=d(x,A)<r\leq\tau\leq d(a,x)=s$ and,
  accordingly,
  \[
    \varphi(a)\cdot \mathbf{F}^*(a,x)=\varphi(a)\cdot
    \mathbf{F}(s,t)\leq\varphi(a)\cdot\frac\varepsilon\lambda\leq
    \varepsilon\leq\varphi(p)+\varepsilon.
  \]
  Otherwise, if $d(a,p)<2\tau$, then
  $\varphi(a)\cdot \mathbf{F}^*(a,x)\leq \varphi(a)\cdot
  1<\varphi(p)+\varepsilon$. Thus,
  $ \Omega_\mathbf{F}[\varphi](x)=\sup_{a\in A} \varphi(a)\cdot
  \mathbf{F}^*(a,x)\leq \varphi(p)+\varepsilon$.\medskip
  
  For the left-hand side of \eqref{eq:special-fn-Tietze-v4:2}, we will
  use \eqref{eq:special-fn-Tietze-v6:1} that
  $\mathbf{F}(t,t)\xrightarrow[t\to 0]{}1$. Since
  $0<\varepsilon<\lambda$, there exists $0<\delta\leq r$ such that
  $\mathbf{F}(t,t)>1-\frac\varepsilon\lambda$ for every
  $0<t<\delta$. If $\varphi(p)-2\varepsilon\leq0$, then
  $\Omega_\mathbf{F}[\varphi](x)\geq 0\geq \varphi(p)-2\varepsilon$
  for every $x\in X$. Suppose that $\varphi(p)-2\varepsilon>0$. Then
  $
  1-\frac\varepsilon\lambda>1-\frac\varepsilon{\varphi(p)-\varepsilon}=
  \frac{\varphi(p)-2\varepsilon}{\varphi(p)-\varepsilon}$ because
  $0<\varphi(p)-\varepsilon<\lambda$. Accordingly, for a point
  $x\in \mathbf{O}(p,\delta)\setminus A\subset \mathbf{O}(p,\tau)$ and
  $t=d(x,A)<\delta$, it follows from Proposition
  \ref{proposition-special-fn-Tietze-v5:1} that
  \begin{align*}
    \Omega_\mathbf{F}[\varphi](x)
    &\geq \sup_{a\in \mathbf{O}_A(p,2\tau)}\varphi(a)\cdot \mathbf{F}^*(a,x)
      \geq [\varphi(p)-\varepsilon]\cdot \sup_{a\in
      \mathbf{O}_A(p,2\tau)}\mathbf{F}^*(a,x)\\ 
    &=[\varphi(p)-\varepsilon]\cdot \mathbf{F}(t,t)>
      [\varphi(p)-\varepsilon]\cdot
      \frac{\varphi(p)-2\varepsilon}{\varphi(p)-\varepsilon}=
      \varphi(p)-2\varepsilon.  
  \end{align*}
  In other words,
  $\Omega_\mathbf{F}[\varphi](x)\geq \varphi(p)-2\varepsilon$ for
  every $x\in \mathbf{O}(p,\delta)\setminus A$. Since $\delta\leq r$,
  this shows \eqref{eq:special-fn-Tietze-v4:2} and the proof is
  complete.
\end{proof}

\begin{proof}[Proof of Theorem \ref{theorem-special-fn-v2:1}]
  Let $\mathbf{F}:\Delta\to (0,1]$ be a Tietze extender,
  $\varphi\in C_+^*(A)$ and
  $f=\Omega_\mathbf{F}[\varphi]\in \ell_\infty^+(X)$ be the extension
  defined in \eqref{eq:special-fn-Tietze-v5:1}. According to
  Propositions \ref{proposition-Mandelkern:1} and
  \ref{proposition-UC-Extensions-rv2:1}, it suffices to show that for
  each $\tau>0$, the restriction of $f$ on
  ${X\setminus \mathbf{O}(A,\tau)}$ is uniformly continuous.  So, take
  $\tau>0$ and $\varepsilon>0$. Since $\mathbf{F}$ satisfies
  \eqref{eq:special-fn-Tietze-v6:3}, by Proposition
  \ref{proposition-special-fn-Tietze-v6:2}, there exists $\delta>0$
  such that \eqref{eq:special-fn-Tietze-v9:1} holds for the function
  $\mathbf{F}^*$. Accordingly, for $x,p\notin \mathbf{O}(A,\tau)$ with
  $d(x,p)<\delta$, it follows from \eqref{eq:special-fn-Tietze-v9:2}
  and \eqref{eq:special-fn-Tietze-v9:1} that
  $|f(x)-f(p)| \leq \sup_{a\in A}\varphi(a) \left|\mathbf{F}^*(a,x)-
    \mathbf{F}^*(a,p)\right| \leq \varepsilon\cdot \|\varphi\|$.
  Thus, $f\uhr X\setminus\mathbf{O}(A,\tau)$ is uniformly continuous
  and the proof is complete.
\end{proof}

Let $E$ and $V$ be Riesz spaces, and $\Omega:E^+\to V^+$ be a
positively homogeneous additive map, i.e.\ a positively homogeneous
map with the property that
\[
  \Omega(u+v)=\Omega(u)+\Omega(v),\quad \text{for every $u,v\in E^+$.}
\]
Then it has a unique extension to a positive linear map
$\Lambda:E\to V$. This is a classical result and can be found in
several books on Riesz spaces, see for instance \cite[Extension Lemma
on p.\ 51]{Gloden1966} and \cite[Lemma 2.10]{Jonge1977}, also
\cite[Lemma 20.1]{Zaanen1997}. Briefly, each $u\in E$ can be
represented in the form $u=u^+-u^-$, where
$u^+=u\vee \mathbf{0}\in E^+$ and $u^-=(-u)\vee \mathbf{0}\in
E^+$. This representation is minimal in the sense that $u=v-w$ for
$v,w\in E^+$ implies that $u^+\leq v$ and $u^-\leq w$. The extension
$\Lambda$ of $\Omega$ is now simply defined by
$\Lambda(u)= \Omega(u^+)- \Omega(u^-)$, $u\in E$.  Based on the same
idea, we will extend the operator
$\Omega_\mathbf{F}:\ell_\infty^+(A)\to \ell_\infty^+(X)$ to the entire
space $\ell_\infty(A)$. Namely, if $\mathbf{F}:\Delta\to (0,1]$ and
$\Omega_\mathbf{F}$ is as in \eqref{eq:special-fn-Tietze-v5:1}, then
we may define an extension operator
$\Theta_\mathbf{F}:\ell_\infty(A)\to \ell_\infty(X)$ by
\begin{equation}
  \label{eq:UC-Extensions-v8:1}
  \Theta_\mathbf{F}[\varphi]= \Omega_\mathbf{F}[\varphi^+]-
  \Omega_\mathbf{F}[\varphi^-], \quad \varphi\in \ell_\infty(A).
\end{equation}

In the theorem below, we say that
$\Theta:\ell_\infty(A)\to \ell_\infty(X)$ \emph{preserves the norm} if
$\|\Theta[\varphi]\|=\|\varphi\|$, for every
$\varphi\in \ell_\infty(A)$.

\begin{theorem}
  \label{theorem-UC-Extensions-v22:1}
  Let $\mathbf{F}:\Delta\to (0,1]$ be a Tietze extender, and
  $\Theta_\mathbf{F}:\ell_\infty(A)\to \ell_\infty(X)$ be the
  extension operator defined in \eqref{eq:UC-Extensions-v8:1}. Then
  $\Theta_\mathbf{F}$ is a $2$-Lipschitz map which is positive and
  positively homogeneous. Moreover, it preserves the norm and both
  continuity and uniform continuity.
\end{theorem}

\begin{proof}
  If $\varphi\in \ell_\infty(A)$ is (uniformly) continuous, then so
  are $\varphi^+$ and $\varphi^-$. Hence, by
  \eqref{eq:UC-Extensions-v8:1} and Theorem
  \ref{theorem-special-fn-v2:1}, $\Theta_\mathbf{F}$ preserves both
  continuity and uniform continuity. For the same reason,
  $\Theta_\mathbf{F}$ is positive and positively homogeneous. Take
  $\varphi\in \ell_\infty(A)$ and $p\in X\setminus A$. Then by
  \eqref{eq:special-fn-Tietze-v9:2}, \eqref{eq:special-fn-Tietze-v5:1}
  and \eqref{eq:UC-Extensions-v8:1},
  \begin{align*}
    |\Theta_\mathbf{F}[\varphi](p)|
    &=|\Omega_\mathbf{F}[\varphi^+](p)-
      \Omega_\mathbf{F}[\varphi^-](p)|\\
    &\leq \sup_{a\in
      A}|\varphi^+(a)-\varphi^-(a)|\cdot\mathbf{F}^*(a,p)\\
    &=\sup_{a\in
      A}|\varphi(a)|\cdot \mathbf{F}^*(a,p)\leq \sup_{a\in
      A}|\varphi(a)|=\|\varphi\|.   
  \end{align*}
  Accordingly, $\Theta_\mathbf{F}$ preserves the norm as
  well. Finally, we show that $\Theta_\mathbf{F}$ is
  $2$-Lipschitz. Indeed, if $\varphi,\psi\in \ell_\infty(A)$ and
  $a\in A$, then $|\varphi^+(a)-\psi^+(a)|\leq |\varphi(a)-\psi(a)|$
  and $|\varphi^-(a)-\psi^-(a)|\leq |\varphi(a)-\psi(a)|$.  Hence, by
  \eqref{eq:UC-Extensions-v8:1} and Theorem
  \ref{theorem-special-fn-v2:1},
  \[
    \left\|\Theta_\mathbf{F}[\varphi]-\Theta_\mathbf{F}[\psi] \right\|
    \leq \|\varphi^+-\psi^+\| + \|\varphi^--\psi^-\|\leq
    2\|\varphi-\psi\|.\qedhere
  \]
\end{proof}

We conclude with several remarks.

\begin{remark}
  \label{remark-UC-Extensions-v22:3}
  The extension operator $\Theta_\mathbf{F}$ defined in
  \eqref{eq:UC-Extensions-v8:1} is not necessarily an isometry. For
  instance, take $X=[-1,+\infty)\subset \R$ and $A=[-1,1]$. Then for
  $a\in A$ and $p\in X\setminus A$ we have that $d(p,a)= p-a$ and
  $d(p,A)=p-1$. Let $\mathbf{F}(s,t)=\mathbf{R}(s,t)=\frac ts$ be the
  Riesz function in \eqref{eq:special-fn-Tietze-v3:4}, which is a
  Tietze extender, see Example \ref{example-UC-Extensions-v7:1}. If
  $\varphi:A\to \R$ is defined by $\varphi(a)=\frac{3a+1}4$, then one
  can easily see that
  $\Theta_\mathbf{R}[\varphi](p)= \frac{p+3}{2(p+1)}$, whenever
  $p\in X\setminus A$. Similarly, if $\psi:A\to \R$ is defined by
  $\psi(a)=\frac{3a-1}4$, then
  $\Theta_\mathbf{R}[\psi](p)= \frac{3-p}{2(p+1)}$, for
  $p\in X\setminus A$. In fact, $\varphi$ and $\psi$ are uniformly
  continuous functions with $\psi=\varphi-\frac12$, so
  $\|\varphi-\psi\|=\frac12$. However, for $p\in X\setminus A$, we
  have that
  $|\Theta_\mathbf{R}[\varphi](p)-\Theta_\mathbf{R}[\psi](p)| = \frac
  p{p+1}\xrightarrow[p\to +\infty]{}1$.\hfill\qed
\end{remark}

\begin{remark}
  \label{remark-UC-Extensions-v17:1}
  Theorem \ref{theorem-special-fn-v2:1} is not valid for bounded
  functions $\varphi:A\to \R$ which may take negative values. For
  instance, let $X=\R$ and $A=(-\infty,0]$. Also, let
  $\mathbf{R}(s,t)=\frac ts$, $s\geq t>0$, be the Riesz function in
  \eqref{eq:special-fn-Tietze-v3:4} of Example
  \ref{example-UC-Extensions-v7:1}, and $-1_A:A\to \{-1\}$ be the
  constant function $-1$. Then for $x\in X\setminus A$ and $a\in A$,
  we have that $\mathbf{R}^*(a,x)=\frac x{x-a}=\frac x{x+|a|}$ and,
  accordingly,
  \[
    \Omega_\mathbf{R}[-1_A](x)= \sup_{a\in A}\left[
      -\mathbf{R}^*(a,x)\right]= -\inf_{a\in A} \mathbf{R}^*(a,x)=
    -\inf_{a\in A}\frac x{x+|a|}=0.\hfill\qed
  \]
\end{remark}

\begin{remark}
  \label{remark-UC-Extensions-v18:1}
  In contrast to Tietze extenders, the extension construction in
  \eqref{eq:UC-Extensions-v11:1} doesn't preserve continuity for
  bounded functions $\varphi:A\to [0,+\infty)$ which may take the
  value $0$. Indeed, let $A=[-1,0]\subset [-1,+\infty)=X$ and
  $\varphi:A\to [0,+\infty)$ be defined by $\varphi(a)=1+a$, $a\in
  A$. If $\mathbf{G}:\Delta\to [1,+\infty)$ is any function, then the
  extension $\mho_\mathbf{G}[\varphi]:X\to [0,+\infty)$ is not
  continuous because $\mho_\mathbf{G}[\varphi](0)=\varphi(0)=1$ and
  $\mho_\mathbf{G}[\varphi](p)=0$ for every $p\in X\setminus A$. \qed
\end{remark}

\section{Bohr's Extension Operator}
\label{sec:bohrs-extens-constr}

As mentioned in the Introduction, the book of Carath\'{e}odory
\cite{caratheodory:1918} contains another proof of Tietze's extension
theorem (Theorem \ref{theorem-special-fn-Tietze-v3:1}), it is credited
to Harald Bohr. In fact, Bohr's extension construction gives at once
an extension operator $\Phi:\ell_\infty(A)\to \ell_\infty(X)$ which
has the best properties comparing with the previous extension
operators. In this construction, for convenience, let
$\rho=d(\cdot,A)$ be the distance function to the set $A$. Whenever
$\varphi\in \ell_\infty(A)$ and ${x}\in X\setminus A$, Bohr associated
the bounded increasing function $\eta_{x,\varphi}:(0,+\infty)\to \R$
defined by
\begin{equation}
  \label{eq:various-ext:3}
  \eta_{x,\varphi}(t)=
  \begin{cases}
    \sup_{{a}\in \mathbf{O}_A({x},t)}\varphi({a})
    &\text{if $t>\rho(x)$, and}\\
    \inf \varphi &\text{if $t\leq \rho(x)$.}
  \end{cases}
\end{equation}
Next, using these functions, he defined an extension
$\Phi[\varphi]\in \ell_\infty(X)$ of $\varphi$ by the following
explicit formula:
\begin{equation}
  \label{eq:various-ext:4}
  \Phi[\varphi]({x})=
  \frac1{\rho({x})}\int_{\rho({x})}^{2\rho({x})} \eta_{x,\varphi}(t)\,
  dt,\quad 
  \text{${x}\in X\setminus A$.}
\end{equation}
Since $\eta_{x,\varphi}$ is bounded and increasing, it is Riemann
integrable. Hence, $\Phi[\varphi]$ is well defined. Also, by the
properties of the integral, we get at once that Bohr's extension
operator $\Phi:\ell_\infty(A)\to \ell_\infty(X)$ is isotone and, in
particular, positive. Furthermore, $\Phi$ is sublinear, i.e.\ both
positively homogeneous and subadditive. Indeed, if $\lambda\geq0$ and
$\varphi\in\ell_\infty(A)$, then by \eqref{eq:various-ext:4},
$\Phi[\lambda\varphi]=\lambda\Phi[\varphi]$ because
${\eta_{x,\lambda\varphi}=\lambda \eta_{x,\varphi}}$ for every
$x\in X\setminus A$, see \eqref{eq:various-ext:3}.  Similarly, for
$\varphi,\psi\in \ell_\infty(A)$, we get that
$\Phi[\varphi+\psi]\leq \Phi[\varphi]+\Phi[\psi]$ because
$\eta_{x,\varphi+\psi}(t)\leq \eta_{x,\varphi}(t)+\eta_{x,\psi}(t)$,
for every $x\in X\setminus A$ and $t>\rho(x)$. Finally, it is also
easy to see that $\Phi$ is an isometry. Namely, for
$\varphi,\psi\in\ell_\infty(A)$ and $x\in X\setminus A$, it follows
from \eqref{eq:special-fn-Tietze-v9:2} and \eqref{eq:various-ext:3}
that
$\left|\eta_{x,\varphi}(t)-\eta_{x,\psi}(t)\right|\leq
\|\varphi-\psi\|$ for every $t\geq\rho(x)$. Accordingly,

\begin{align*}
  |\Phi[\varphi](x)-\Phi[\psi](x)|
  &\leq \frac1{\rho({x})}\int_{\rho({x})}^{2\rho({x})}
    \left|\eta_{x,\varphi}(t) - 
    \eta_{x,\psi}(t)\right| dt\\
  &\leq \frac1{\rho({x})}\int_{\rho({x})}^{2\rho({x})}
    \|\varphi - 
    \psi\| dt=\|\varphi-\psi\|.
\end{align*}

Regarding other properties of $\Phi:\ell_\infty(A)\to \ell_\infty(X)$,
Bohr actually showed that it preserves continuity. The interested
reader is also referred to \cite{Friedman1982} and \cite{Arenas2002},
where these arguments were reproduced. Below we show that $\Phi$
preserves uniform continuity as well.

\begin{theorem}
  \label{theorem-special-fn-Tietze-v11:1}
  The extension operator $\Phi:\ell_\infty(A)\to \ell_\infty(X)$
  defined as in \eqref{eq:various-ext:4} is a sublinear isotone
  isometry which preserves both continuity and uniform continuity.
\end{theorem}

\begin{proof}  
  Let $\varphi\in \ell_\infty(A)$ be a (uniformly) continuous
  function, and $f=\Phi[\varphi]$ be the associated extension in
  \eqref{eq:various-ext:4}. If $\lambda>0$, $\mu\in \R$ and
  $x\in X\setminus A$, then it follows from \eqref{eq:various-ext:3}
  that ${\eta_{x,\lambda\varphi +\mu}=\lambda
    \eta_{x,\varphi}+\mu}$. Hence, by \eqref{eq:various-ext:4},
  $\Phi[\lambda\varphi+\mu]=\lambda f+\mu$. Accordingly, we may assume
  that $\varphi:A\to [0,1]$. Moreover, in the rest of this proof we
  will simply write $\eta_x$ instead of $\eta_{x,\varphi}$.\smallskip

  First, we will show that $f$ is continuous at the points of $A$,
  also that it is strongly uniformly continuous on $A$ provided
  $\varphi$ is itself uniformly continuous. To this end, take
  $\varepsilon>0$ and $p\in A$. As in the previous proofs, using that
  $\varphi$ is continuous at $p$, there is $\delta>0$ (independent of
  $p$ provided that $\varphi$ is uniformly continuous) such that
  $|\varphi({a})-\varphi({p})|<\varepsilon$ for every
  $a\in \mathbf{O}_A(p,3\delta)$.  If
  ${x}\in \mathbf{O}({p},\delta)\setminus A$, then
  $\mathbf{O}({x},2\rho({x}))\subset \mathbf{O}({p},3\delta)$ because
  $\rho({x})\leq d({x},{p})<\delta$. Accordingly, see
  \eqref{eq:various-ext:3} and~\eqref{eq:various-ext:4},
  \begin{align*}
    \varphi({p})-\varepsilon\leq \inf_{{a}\in
    \mathbf{O}_A({p},3\delta)}&\varphi({a})
                                \leq \inf_{\rho(x)<t\leq 2\rho(x)} \eta_x(t)\\
                              &\leq
                                f({x})\leq \eta_{x}(2\rho({x}))\leq
                                \sup_{{a}\in \mathbf{O}_A({p},3\delta)}
                                \varphi({a})
                                \leq
                                \varphi({p})+\varepsilon. 
  \end{align*}
  Thus, $|f({x})-\varphi({p})|\leq \varepsilon$.\medskip

  Let $\tau>0$. We finalise the proof by showing that
  $f\uhr X\setminus \mathbf{O}(A,\tau)$ is uniformly continuous, in
  fact Lipschitz. So, let ${x},{p}\in X\setminus \mathbf{O}(A,\tau)$
  with $0<d(x,p)\leq\frac\tau3$. Then
  $|\rho({x})-\rho({p})|\leq d({x},{p})=\delta\leq\frac{\rho({x})}3$
  and, therefore,
  \[
    \rho({p})\leq \rho({x})+\delta \leq 2\rho({x})-2\delta\leq
    2\rho({p}).
  \]
  Moreover, $\eta_{p}(t)\geq \eta_{x}(t-\delta)$ for every $t>\delta$
  because $\mathbf{O}({x},t-\delta)\subset \mathbf{O}({p},t)$, see
  \eqref{eq:various-ext:3}. Thus, substituting in
  \eqref{eq:various-ext:4}, we get that
  \begin{align*}
    f({p})=\frac1{\rho({p})}\int_{\rho({p})}^{2\rho({p})}
    \eta_{p}(t) dt 
    &\geq
      \frac1{\rho({x})+\delta}
      \int_{\rho({x})+\delta}^{2\rho({x})-2\delta} 
      \eta_{x}(t-\delta) dt\\
    (s=t-\delta)\qquad &=
                         \frac1{\rho({x})+\delta}
                         \int_{\rho({x})}^{2\rho({x})-3\delta}  
                         \eta_{x}(s) ds.
  \end{align*}
  Since
  $f(x)= \frac1{\rho(x)}\int_{\rho({x})}^{2\rho({x})-3\delta}
  \eta_{x}(t)dt +\frac1{\rho(x)}\int_{2\rho({x})-3\delta}^{2\rho({x})}
  \eta_{x}(t)dt$ and $\eta_{x}\leq 1$, this implies that
  \begin{eqnarray*}
    f({x})-f({p})
    &\leq& \left[\frac1{\rho({x})}-
           \frac1{\rho({x})+\delta}\right] 
           \int_{\rho({x})}^{2\rho({x})-3\delta} \eta_{x}(t)dt +
           \frac1{\rho({x})}
           \int_{2\rho({x})-3\delta}^{2\rho({x})}
           \eta_{x}(t)dt\\ 
    &=& \frac\delta{\rho({x})[\rho({x})+\delta]}
        \int_{\rho({x})}^{2\rho({x})-3\delta} \eta_{x}(t)dt +
        \frac1{\rho({x})}\int_{2\rho({x})-3\delta}^{2\rho({x})}
        \eta_{x}(t)dt\\
    &\leq& \frac{\delta}{\rho(x)}\cdot \frac{\rho({x})-3\delta}{
           \rho({x})+\delta} + 
           \frac{3\delta}{\rho({x})}\leq 
           \frac{4\delta}{\rho({x})}\leq 
           \frac{4}{\tau} d(x,p).   
  \end{eqnarray*}
  Interchanging ${p}$ and ${x}$, this is equivalent to
  $|f ({p})- f ({x})|\leq \frac4\tau d({x}, {p})$. In particular, $f$
  is continuous and by Proposition \ref{proposition-Mandelkern:1}, it
  is also uniformly continuous whenever so is $\varphi$.
\end{proof}

\subsection*{Acknowledgements}
The author is very grateful to both referees for their helpful
suggestions and remarks.

\providecommand{\bysame}{\leavevmode\hbox to3em{\hrulefill}\thinspace}

\end{document}